\documentclass[12pt]{article}
\usepackage{amssymb}

\usepackage{amsthm}

\usepackage{ifpdf}
\usepackage{graphicx}
\usepackage{url}
\usepackage{hyperref}
\usepackage{verbatim}
\usepackage{latexsym}
\usepackage{amscd}
\usepackage{amsmath}
\usepackage{amssymb}
\usepackage{amsfonts}
\usepackage{amsrefs}
\usepackage[usenames,dvipsnames,table]{xcolor} 
\usepackage{enumerate}
\usepackage[nameinlink,sort]{cleveref}
\usepackage{rotating}

\allowdisplaybreaks

\newcommand{\rank}{\text{rank}}

\newcommand{\CA}{{\rm CA}}
\newcommand{\CAN}{{\rm CAN}}

\newcommand{\VCA}{{\rm VCA}}
\newcommand{\VCAN}{{\rm VCAN}}

\newcommand{\Z}{\mathbb{Z}}

\newcommand{\vardens}{{\sc VarDens\ }}

\newcommand{\ignore}[1]{}

\numberwithin{equation}{section}

\makeatletter
\def\imod#1{\allowbreak\mkern10mu({\operator@font mod}\,\,#1)}
\makeatother
\numberwithin{equation}{section}

\newcommand{\rmv}[1]{}

\def\CA{\mathrm{CA}}

\def\CAN{\mathrm{CAN}}
\def\Z{\mathbb{Z}}

\newtheorem{theorem}{Theorem}[section]
  \newtheorem{lemma}[theorem]{Lemma}
  \newtheorem{definition}[theorem]{Definition}

\begin{document}

\title{Upper bounds on the sizes of variable strength covering arrays
using the Lov\'{a}sz local lemma}

\author{Lucia Moura\footnote{University of Ottawa, Ottawa, Canada, \href{mailto:lmoura@uottawa.ca}{\url{lmoura@uottawa.ca}}}, 
Sebastian Raaphorst\footnote{Gemini Observatory, La Serena, Chile, \href{mailto:sraaphorst@gmail.com}{\url{sraaphorst@gmail.com}}}
and Brett Stevens\footnote{Carleton University, Ottawa, Canada, \href{mailto:brett@math.carleton.ca}{\url{brett@math.carleton.ca}}} }

\date{January 15, 2019}
\maketitle

\begin{abstract}
Covering arrays are generalizations of orthogonal arrays that have been widely studied
and are used in software testing. The probabilistic method has been employed to derive 
upper bounds on the sizes of minimum covering arrays and give asymptotic upper
bounds that are logarithmic on the number of columns of the array.
This corresponds to test suites with a desired level of coverage of the parameter space
where we guarantee the number of test cases is logarithmic on the number of parameters of the system.
In this paper, we study variable strength covering arrays, a generalization of covering
arrays that uses a hypergraph to specify the sets of columns where coverage is required;
(standard) covering arrays is the special case where coverage is required for all sets of
columns of a fixed size $t$, its strength.
We use the probabilistic method to obtain upper bounds
on  the number of rows of a variable strength covering array, given in terms of parameters
of the hypergraph. We then compare this upper bound with
another one given by a density-based greedy algorithm on different types of hypergraph
such as $t$-designs, cyclic consecutive hypergraphs,
planar triangulation hypergraphs, and a more specific hypergraph given by a clique of higher 
strength on top of a ``base strength". The conclusions are dependent on the class of hypergraph,
and we discuss specific characteristics of the hypergraphs which are more amenable to using 
different versions of the Lov\'{a}sz local lemma.

\end{abstract}






\section{Introduction}\label{intro}

Covering arrays are well studied combinatorial designs~\cite{colbourn04,crchandbook} in part because of their utility in software and network testing \cites{cohen96,dalal98,kuhn04}.  For more information about  covering arrays, including combinatorial constructions and overview on algorithms and applications, see the survey by Colbourn~\cite{colbourn04}.
In this paper, we 
focus on a recent covering array generalization called variable strength covering array (VCA).
We  begin defining VCAs and indicate how covering arrays are a special case. For more information on VCAs, see \cite{raaphorstphd,vca}. 
\begin{definition}
Let $H=(V,E)$ be a hypergraph and let $k = |V|$. A {\em variable-strength covering array}, denoted $\VCA(n; H, v)$, is an $n \times k$ array  filled from $\Z_v$ such that for each $e = \{v_0, \ldots, v_{t-1}\} \in E$, the $n \times t$ subarray of columns indexed by the elements of $e$ is {\em covered}, that is, it has every possible $t$-tuple in $\Z_v$ as a row at least once. The {\em variable-strength covering array number}, written $\VCAN(H, v)$, is the smallest $n$ such that a $\VCA(n; H, v)$ exists.
\end{definition}
Consider the complete $t$-uniform hypergraph on $k$ vertices, denoted $K_{k}^{(t)}$, that is, 
the hypergraph where $|V|=k$ and $E$ is the set of all $t$-subsets of $V$. 
A covering array of strength $t$, denoted $\CA(N;t,k,v)$, is precisely a  $\VCA(N;K_{k}^{(t)},v)$; the  covering array number is denoted $\CAN(t,k,v)$. 
In this article, we use the Lov\'{a}sz local lemma to give an upper  bound on $\VCAN(H,v)$.
\begin{theorem}[Lov\'{a}sz local lemma - symmetric case \cites{lovasz75,spencer77}]
\label{thm-prob-symlocallemma}
\index{local lemma!symmetric|textbf}
Let $\mathcal{A} = \{A_0, \ldots, A_{m-1}\}$ be a finite set of events in a probability space $\Omega$ such that $\Pr(A_i) \leq p < 1$, and each event is independent of all but at most $d$ of the other events. If $ep(d+1) \le 1$, where $e$ is the base of the natural logarithm, then the probability that none of the events occur is nonzero.
\end{theorem}
For standard covering arrays, the local lemma has been used \cite{godbole96} to prove 
\[\CAN(t,k,v) \le \frac{(t-1) \log k}{\log \frac{v^t}{v^t-1}} (1+o(1)).\]
More recently, this local lemma technique and the deterministic analogue, often called {\em entropy compression} \cite{moser2010}, has improved the coefficient of the leading term in this bound for all $v$ when $t=3$ and for $v \leq 10$  when $4 \leq t \leq 6$, by using columns that have balanced numbers of symbols \cites{MR3328867,francetic2016}.  Sarkar and Colbourn match and extend these improvements by exploiting permutation groups with the use of the local lemma \cite{sarkar_upper_2017}. They also combine the local lemma, permutation groups, graph colouring techniques and a density-based greedy approach into two-stage methods which further improve covering array upper bounds \cite{colbourn_two-stage_nodate_alt}. In \cite{colbourn_asymptotic_nodate} together with Lanus they use the local lemma to construct {\em covering perfect hash families} which are then used to construct covering arrays; covering perfect hash families are smaller than covering arrays and thus can be more efficiently generated.  These constructions are related to covering arrays constructed from linear feedback shift register sequences \cites{raaphorst2012,MR3433921}.

Probabilistic methods have been previously used in the context of
variable strength covering arrays.  In
\cites{godbole2011,godbole2010cca} probabilistic methods, but not
explicitly the local lemma, are used to derive bounds on variable
strength covering arrays for {\em consecutive hypergraphs} which are
similar to the cyclic consecutive hypergraphs we discuss in
Section~\ref{sec-cc-triangle} but without the edges wrapping from the
end of the vertex set to the beginning. Sarkar et al. \cite{MR3808633} use the local lemma to construct {\em partial covering arrays} which cover at least $1-\epsilon$ of the possible $t$-sets of columns.  The main difference between these and variable strength covering arrays is that variable strength covering arrays specify exactly which $t$-sets of columns will be covered and partial covering arrays only specify that a certain proportion of $t$-sets must be covered, without constraining which ones they are.

\subsection{Main Result}

For a hypergraph $H$, let the rank of $H$, denoted $rank(H)$, be the largest cardinality of an edge in $H$.
\begin{theorem}
\label{theorem-main}
Let $H=(V,E)$ be a hypergraph with $\rank(H) = t\geq 1$, and let $d$ be an integer such that no edge of $H$ intersects more than $d$ other edges of $H$. Then, for any $v \ge 2$, we have:
\begin{eqnarray}
\VCAN(H, v) &\le& \left\lceil \frac{\ln(d+1) + t \ln v + 1}{\ln \frac{v^t}{v^t-1}}\right\rceil \label{eqprob}\\
&\le& \left\lceil  v^t \left(\ln(d+1) + t \ln v + 1\right)\right\rceil.
\end{eqnarray}
\end{theorem}
\begin{proof}
\ignore{
Let $k = |V|$. Fill an $N \times k$ array, $M$, uniformly at random from $\Z_v$. For each edge $e$, let $A_e$ be the event that the subarray over the columns of $e$ is not covered. The probability that this happens is bounded above by
$p = v^t\left(\frac{v^t-1}{v^t}\right)^N.$
Applying  Theorem~\ref{thm-prob-symlocallemma} we obtain that if
$
  N \geq \frac{\ln(d+1) + t \ln v + 1}{\ln \frac{v^t}{v^t-1}},
 $
 then a $VCA(N,H,v)$ exists.  This gives the first inequality;  the second one is derived by truncating the Taylor expansion of the denominator.}
Let $k = |V|$. Consider a randomly generated $N \times k$ array $M$ with entries chosen from $\Z_v$ with uniform probability. For each edge $e \in E$, write $s = |e|$, and associate an event $A_e$ that the $N \times s$ subarray of $M$ formed by the columns corresponding to $e$ is missing one or more of the $v^s$ $s$-tuples of $\Z_v^s$ as a row. Define:
\[p = v^t\left(\frac{v^t-1}{v^t}\right)^N.\]
Since $s \leq t$, we have that
\[\Pr(A_e) \le v^s\left(\frac{v^s-1}{v^s}\right)^N \le v^t\left(\frac{v^t-1}{v^t}\right)^N = p.\]

We apply the symmetric Local Lemma as in Theorem~\ref{thm-prob-symlocallemma}, which states that if $ep(d+1) \le 1$, the probability that none of the events occur is positive, and hence there is some $N \times k$ array that is a $\VCA(N; H, v)$. This happens when:
\begin{equation*}
N \geq \frac{\ln(d+1) + t \ln v + 1}{\ln \frac{v^t}{v^t-1}}.
\end{equation*}
Thus, we have:
\[\VCAN(H, v) \le \frac{\ln(d+1) + t \ln v + 1}{\ln \frac{v^t}{v^t-1}}.\]
If we use the approximation
\[
\left(\ln \frac{m}{m-1}\right)^{-1} < m,
\]
for $m > 1$ from the Taylor series expansion, we can rewrite the equation as follows:
\[\VCAN(H, v) \le v^t\left(\ln(d+1) + t \ln v + 1\right).\]

\end{proof}

Theorem~\ref{theorem-main} can be applied to any hypergraph $H$ and thus gives a very general existence result for VCAs with arbitrary parameters.  
The rest of the paper focuses on comparing the upper bound given by Theorem~\ref{theorem-main} with a constructive upper bound for VCAs obtained by a density-based greedy algorithm 
called \vardens introduced in~\cite{raaphorstphd,vcadbga} and given next.
This method is a generalization of the density method of Bryce and Colbourn~\cite{bryce09} for dealing with variable strength.
\begin{theorem} \label{thm-vcadbga-log}~\cite{raaphorstphd,vcadbga}
Let $H$ be a hypergraph  such that  $rank(H)=t$ and $e=|E(H)|$. 
Algorithm \vardens  returns a $\VCA(N; H, v)$ where $N$ satisfies
\begin{eqnarray}
\VCAN(H, v) \leq N &\leq& \left\lceil \frac{\ln e + t\ln v}{\ln \frac{v^t}{v^t-1}} \right \rceil \label{eqdens}\\ 
&\leq& \left\lceil v^t (\ln e + t\ln v)\right \rceil.
\end{eqnarray}
\end{theorem}

\ignore{ 
To run this algorithm can require the computation and storage of a large amount of information.  Sarkar and Colbourn adapt the method to a 2-stage algorithm involving a completely random first stage followed by a deterministic second stage and show how this reduces computation and storage significantly \cite{sarkar_upper_2017}.  They also combine this approach with the local lemma, permutation groups, graph colouring techniques and covering perfect hash families \cites{colbourn_two-stage_nodate_alt, colbourn_asymptotic_nodate}.
}

For the remainder of the paper we compare the bounds given by the probabilistic method and by \vardens algorithm.
Let $N_{prob}(H,v)$ be the upper bound given by the probabilistic method (the right hand side of equation~(\ref{eqprob}) in Theorem~\ref{theorem-main}), and let $N_{dens}(H,v)$  be the upper bound given by \vardens algorithm (the 
right hand side of equation~(\ref{eqdens}) in Theorem~\ref{thm-vcadbga-log}). That is, we denote
\begin{eqnarray*}
N_{prob}(H,v) &:=& \left\lceil \frac{\ln(d+1) + t \ln v + 1}{\ln \frac{v^t}{v^t-1}}\right\rceil,\\
N_{dens}(H,v) &:=& \left\lceil \frac{\ln e + t\ln v}{\ln \frac{v^t}{v^t-1}} \right \rceil.
\end{eqnarray*}

We note that if absolutely nothing is known about the hypergraph $H$ except the number of edges $e$ and the rank $t$, then we can substitute $d \leq e-1$, into $N_{prob}(H,v)$  obtaining an upper bound quite close to $N_{dens}(H,v)$,
specially as we fix $v$ and $t$ and let $e$ grow.

In the rest of the article, we show that when we know better estimates on $d$, $N_{prob}(H,v)$ outperforms $N_{dens}(H,v)$.
In Section~\ref{sec-designs}, we look at hypergraphs that are combinatorial designs.  
In Section~\ref{sec-cc-triangle}, we study two families of hypergraphs (cyclic consecutive hypergraphs 
and planar triangulations) where we know the exact $\VCAN$ to assess how well each of these methods perform.
In Section~\ref{sec-gen-ll}, we consider other versions of the Lov\'{a}sz local lemma,
namely the asymmetric and the general cases, and exemplify challenges and benefits to
using them. An extended abstract containing the main results and discussions in Sections~\ref{intro}-\ref{sec-cc-triangle} appeared in~\cite{cai-ea}.

\section{VCAs over designs}
\label{sec-designs}

Combinatorial designs can be used to obtain hypergraphs that have a great deal of regularity, which can be exploited to determine the number of dependent events.
\begin{definition}
  An $s$-$(k,t,\lambda)$ design is a collection $\mathcal{B}$ of $t$-subsets (called blocks) of a $k$-set $V$ with the property that any $s$-subset of points from $V$ appear in exactly $\lambda$ blocks of $\mathcal{B}$.
\end{definition}
 \noindent For a fixed $s$, $t$ and $\lambda$, these designs are known to exist for all sufficiently large $k$ that satisfy the necessary conditions \cite{1401.3665}.

 We begin by counting the number of blocks in a $s$-$(k,t,\lambda)$ design that intersect a fixed block.
\begin{lemma}
\label{inclusion-exclusion}
Let $\mathcal{B}$ be an $s$-$(k,t,\lambda)$ design. Then for any $B \in \mathcal{B}$:
\[d = |\{B' \in \mathcal{B} \setminus \{B\} : B \cap B' \neq \emptyset\}| \le \sum_{i=1}^{2\lfloor\frac{s-1}{2}\rfloor+1} (-1)^{i+1} {t \choose i} \left(\frac{\lambda {k-i \choose s-i}}{{t-i \choose s-i}} - 1\right).\]
When $s=t-1$ and the design has no repeated blocks,
\[d = \sum_{i=1}^{t-1} (-1)^{i+1} {t \choose i} \left(\frac{\lambda {k-i \choose (t-1)-i}}{t-i} - 1\right).\]
Furthermore, for any $1 \le m \le s$, $m$ odd, we can truncate the summation after $m$ terms to derive an upper bound on the summation.
\end{lemma}

\begin{proof}
  Let $K$ be the point set of the design.  To count the number of blocks that intersect $B$ in at least one point we use the inclusion-exclusion principle.  Let $\lambda_{U}$ be the number of blocks that contain a set $U \subseteq K$.  The inclusion-exclusion principle states that the number of blocks that intersect $B \in \mathcal{B}$ in at least one point is precisely
  \[
 d = \sum_{i = 1}^{t} (-1)^{i+1} \sum_{\scriptsize \begin{array}{c}U \subset B\\ |U| = i\end{array}} \left (\lambda_{U} -1 \right )
  \]
  where the 1 is subtracted to not count $B$ itself.  Whenever the outer summation is truncated ending with a positive term an upper bound is achieved.

  For $|U| \leq s$, $\lambda_U$ can be computed from the parameters of the design.  Each of the ${t \choose i}$ $i$-sets $U \subseteq B$, occurs in exactly ${k-i \choose s-i}$ $s$-sets each of which appears $\lambda$ times amongst the blocks of $\mathcal{B}$. Each block of $\mathcal{B}$ which contains $U$, contains exactly ${t-i \choose s-i}$ $s$-sets containing $U$, and hence $\lambda_{U} = \lambda \frac{{k-i \choose s-i}}{{t-i \choose s-i}}$ blocks of $\mathcal{B}$.  For $s+1 \leq |U| \leq t$, $\lambda_U$ cannot be derived from just the parameters of the design.  If there are repeated blocks then $\lambda_{B}$ itself may be more than one.  Thus when $s < t-1$, or repeated blocks are permitted, the value of $d$ will depend on the particular design and not just on the parameters.  So we calculate an upper bound on $d$, by truncating the inclusion-exclusion after a positive term. We stop at the largest odd integer no larger than $s$, that is $2\lfloor \frac{s-1}{2} \rfloor + 1$.

  When $s = t-1$ and there are no repeated blocks, every term in the summation can be computed from the parameters of the design and the computation is exact.
\end{proof}

\begin{theorem}[Bound Comparison for $\VCA$ over $s$-$(k,t,\lambda)$ designs]
\label{thm-prob-designllbound}
\ \\
Let $\mathcal{B}$ be an $s$-$(k,t,\lambda)$ design. 
Then,
\[N_{prob}(\mathcal{B},v)= \left\lceil \frac{\ln \left(1+\sum_{i=1}^{2\lfloor\frac{s-1}{2}\rfloor+1} (-1)^{i+1} {t \choose i} \left(\frac{\lambda {k-i \choose s-i}}{{t-i \choose s-i}} - 1\right) \right ) + t \ln v + 1}{\ln \frac{v^t}{v^t-1}}\right\rceil,\]

\noindent and $N_{dens}(\mathcal{B},v)= \left\lceil \frac{\ln b + t \ln v}{\ln \frac{v^t}{v^t-1}}\right\rceil.$ Furthermore, for $s$, $t$, $\lambda$, and $v$ fixed, as $k \rightarrow \infty$, we get
\[N_{prob}(\mathcal{B},v)=(s-1) v^t \ln k + O(1)\;\; \mbox{ and }\;\;   N_{dens}(\mathcal{B},v)=s v^t \ln k + O(1).\]
\end{theorem}
\begin{proof} The first equality comes from Lemma~\ref{inclusion-exclusion} to bound $d$ in Theorem~\ref{theorem-main}, the second equality comes from Theorem~\ref{thm-vcadbga-log} and the asymptotic results follow easily from these using the Taylor series of the denominator.\end{proof}

We conclude that for designs, $N_{prob}(\mathcal{B},v)$ is asymptotically better than $N_{dens}(\mathcal{B},v)$ as it reduces the coefficient for the leading term $\ln k$
by $v^t$.

When $s = t-1$ and the design has no repeated blocks we can take advantage of the equality in Lemma~\ref{inclusion-exclusion}. Table~\ref{steiner-systems} gives the bounds from Theorem~\ref{thm-prob-designllbound} for $(t-1)$-$(k,t,1)$ designs for $3 \leq t \leq 6$. When $t \leq 6$, explicit designs are known \cite{crchandbook}. Although for small $k$, $N_{dens}$ may outperform $N_{prob}$, as $k$ grows even modestly, $N_{prob}$ becomes better.

\begin{table}[!htbp]
\caption{VCA upper bounds on $(t-1)$-$(k,t,1)$ designs.    \label{steiner-systems}}
  \[ \hspace{-1.5cm}
    {\tiny
    \begin{array}{|c|l|l|l|l|}
      \hline
      & $2$-$(k,3,1)$\mbox{ designs}& $3$-$(k,4,1)$\mbox{ designs} & $4$-$(k,5,1)$\mbox{ designs}& $5$-$(k,6,1)$\mbox{ designs} \\ \hline
      \mbox{ $N_{prob}$} & \frac{\ln\left(\frac{3}{2}k-\frac{7}{2}\right) + 3 \ln v + 1}{\ln \frac{v^3}{v^3-1}} & \frac{\ln\left(\frac{2}{3}k^2 - \frac{11}{3}k + 9\right) + 4 \ln v + 1}{\ln \frac{v^4}{v^4-1}} &       \frac{\ln\left(\frac{5}{24}k^3 - \frac{35}{12}k^2 + \frac{125}{8}k - \frac{121}{4}\right) + 5 \ln v + 1}{\ln \frac{v^5}{v^5-1}} &       \frac{\ln\left(\frac{1}{20}k^4 - \frac{9}{8}k^3 + \frac{257}{24}k^2 - \frac{595}{12}k + \frac{456}{5}\right) + 6 \ln v + 1}{\ln \frac{v^6}{v^6-1}} \\ \hline
      \mbox{$N_{dens}$} & \frac{\ln\left(\frac{1}{6}k^2 - \frac{1}{6}k\right) + 3 \ln v}{\ln \frac{v^3}{v^3-1}}       & \frac{\ln\left(\frac{1}{24}k^3 - \frac{1}{8}k^2 + \frac{1}{12}k\right) + 4 \ln v}{\ln \frac{v^4}{v^4-1}}  & \frac{\ln\left(\frac{1}{120}k^4 - \frac{1}{20}k^3 + \frac{11}{120}k^2 - \frac{1}{20}k\right) + 5 \ln v}{\ln \frac{v^5}{v^5-1}}  & \frac{\ln\left(\frac{1}{720}k^5 - \frac{1}{72}k^4 + \frac{7}{144}k^3 - \frac{5}{72}k^2 + \frac{1}{30}\right) + 6 \ln v}{\ln \frac{v^6}{v^6-1}} \\
      \hline
    \end{array}
    }
  \]
  
\end{table}

\section{Ability of bounds to capture global properties of $H$}
\label{sec-cc-triangle}
In this section, we compare the two bounds for two families of hypergraphs for which VCANs are almost completely known and do not grow with $k$~\cite{vca}. The {\em cyclic consecutive hypergraph} is $H_c^{k,t} = (V,E)$ with $V = \{0, \ldots, k-1\}$ and $E = \{\{i, i+1 \bmod{k}, \ldots, i+t-1 \bmod{k}\} : 0 \le i \le k-1\}$; 
for many cyclic consecutive hypergraphs $\VCAN(H_c^{k,t},v)$ equals  $v^t$, and in all cases we know that $\VCAN(H_c^{k,t},v) \leq d_{t,v}$ for some $d_{t,v}$ that does {\em not} grow with $k$ \cite[Theorem 3.8]{vca}.

In $H_{c}^{k,t}$, each edge intersects exactly $2t-2$ other edges. Theorem~\ref{theorem-main} gives
\begin{equation} \label{ccca}
  N_{prob}(H_c^{k,t},v) = \left\lceil\frac{\ln(2t-1) + t \ln v +  1}{\ln \frac{v^t}{v^t-1}} \right\rceil.
  \end{equation}
In agreement with $d_{t,v}$, $N_{prob}$ is independent of $k$ and thus for fixed $v$ and $t$, it is $O(1)$.  On the other hand, the bound from algorithm \vardens is
\[
  N_{dens}(H_c^{k,t},v)= \left\lceil \frac{\ln k + t \ln v}{\ln \frac{v^t}{v^t-1}}\right\rceil,
  \]
which does grow with $k$.  Thus we have an example of a family of hypergraphs where  $N_{prob}$ is substantially better than  $N_{dens}$. The homomorphism construction (see~\cite[Chapter 3]{raaphorstphd})
gives $VCAN(H_{c}^{k,t},v)\leq \VCAN(H_{c}^{k',t},v)$, where $k' \leq 2t-1$ and the $\ln(2t-1)$ term in Equation~\ref{ccca} in place of the usual $\ln k$ term suggests that $N_{prob}$ ``recognizes'' this homomorphism while $N_{dens}$ does not. However, experiments show that running \vardens algorithm for $H_{c}^{k,t}$ does seem to generate arrays where the array size is independent of $k$ \cite{raaphorstphd}. 

  A {\em triangulation hypergraph of the sphere}, $T=(V,E)$ is a rank-3 hypergraph which corresponds to a planar graph all of whose faces are triangles;  the rank-3 hyperedges are precisely the faces of the planar embedding. From colourings and homomorphisms we know that $\VCAN(T,v)=v^3$~\cite{vca}.
If $T$ has $k$ vertices then  it has  $2k-4$ edges (triangles) and 
  a triangle may intersect up to $3(\Delta-2)$ other triangles where $\Delta$ is the maximum degree of the planar
graph. Thus, $d \leq \min\{3(\Delta-2),2k-5\}$ and 
  \begin{equation*}
    N_{prob}(T,v)\leq \left\lceil\frac{\ln(\min\{3(\Delta-2)+1,2k-4\}) + 3 \ln v + 1}{\ln \frac{v^3}{v^3-1}}\right\rceil,
  \end{equation*}
  while \vardens gives
  \[
    N_{dens}(T,v)=\left\lceil\frac{\ln(2k-4) + 3 \ln v }{\ln \frac{v^3}{v^3-1}}\right\rceil.
  \]
Thus, $N_{prob}$  only performs better than $N_{dens}$ when the maximum degree of the triangulation grows more slowly than the number of vertices.  In addition, unless $\Delta(T)$ is bounded as $k$ grows, $N_{prob}$ grows with $k$, even though $\VCAN(T,v)=v^3$. Hence, triangulation hypergraphs represent a case where, unlike cyclic consecutive hypergraphs, our application of the local lemma is unable to capture the behaviour of the homomorphism $H \rightarrow K_4^{(3)}$. The bound $N_{dens}$ also grows with $k$ and, in contrast to the previous class
of hypergraphs, 
experiments running \vardens on random triangulation hypergraphs suggest that the size of the arrays produced indeed increases as some function of $k$, see Table~\ref{tbl-prob-vcadbgatriangulations}.  For more details about generating the random triangulation hypergraphs see~\cite{raaphorstphd}.

\begin{table}[ht]
\begin{center}
\begin{tabular}{|r|rrr|rrr|rrr|} \hline
 & \multicolumn{3}{|c|}{$k=4$} & \multicolumn{3}{|c|}{$k=11$} & \multicolumn{3}{|c|}{$k=100$}\\ \cline{2-10}
$v$ & $N_m$ & $N_M$ & $N_a$ & $N_m$ & $N_M$ & $N_a$ & $N_m$ & $N_M$ & $N_a$ \\ \hline
2 & 8 & 10 & 8.37 & 8 & 11 & 9.50 & 12 & 14 & 13.13 \\
3 & 27 & 34 & 31.17 & 31 & 36 & 34.23 & 39 & 42 & 40.37 \\
5 & 142 & 152 & 146.30 & 150 & 157 &154.13 & 167 & 172 & 169.60 \\ \hline
\end{tabular}
\caption[\vardens results for triangulations]{\label{tbl-prob-vcadbgatriangulations} Trials of \vardens on random triangulations with $k$ vertices. $N_m$, $N_M$, and $N_a$ are the minimum, maximum, and mean array sizes, respectively.}
\end{center}
\end{table}

\section{Using the general and asymmetric local lemma}
\label{sec-gen-ll}

In Theorem~\ref{theorem-main}, we use the symmetric version of the local lemma to establish the existence of covering arrays.  In Section~\ref{sec-designs}, we apply this theorem to hypergraphs that are highly symmetric: the probability that any set of columns represented by a hyperedge be covered is precisely the same as for any other hyperedge, because all the hyperedges have the same size.  Additionally, for $(t-1)$-$(k,t,\lambda)$ designs the size of the sets of dependent events also does not vary; even for $s$-$(k,t,\lambda)$ designs the sizes of sets of dependent events can only vary in a fixed range.  In Section~\ref{sec-cc-triangle}, the hyperedges were again of a fixed size.  If the hyperedges themselves vary in size or the size of sets of dependent events significantly vary, using the symmetric version of the local lemma requires taking the worst probability of a set of columns being uncovered and requires taking the largest set of dependent events.   In this section, we explore the benefit of using versions of the local lemma that can adapt to varying sizes of hyperedges and sets of dependent events.

The most general statement of the local lemma was given by Lov\'{a}sz in 1975.
\begin{theorem}[Lov\'{a}sz local lemma - general case \cite{lovasz75}]
\label{thm-prob-generallocallemma}
\ \\
Let $\mathcal{A} = \{A_0, \ldots, A_{m-1}\}$ be a finite set of events in a probability space $\Omega$. Define a function $\Gamma\colon \mathcal{A} \rightarrow \mathcal{P}(\mathcal{A})$ such that for $A \in \mathcal{A}$, $A$ is independent from all events in $\mathcal{A} \setminus (\Gamma(A) \cup \{A\})$. If there is a map $x\colon \mathcal{A} \rightarrow (0,1)$ such that for all $A \in \mathcal{A}$:
\[\Pr(A) \le x(A)\prod_{B \in \Gamma(A)} (1 - x(B))\]
then the probability that none of the events occur is nonzero, and is:
\[\Pr(\overline{A_0} \wedge \ldots \wedge \overline{A_{m-1}}) \ge \prod_{A \in \mathcal{A}} (1-x(A)).\]
\end{theorem}
\noindent The Asymmetric local lemma was given by Habib in 1998.
\begin{theorem}[Lov\'{a}sz local lemma - asymmetric case \cite{habib1998}]
\label{thm-prob-asymlocallemma}
\ \\
Let $\mathcal{A} = \{A_0, \ldots, A_{m-1}\}$ be a finite set of events in a probability space $\Omega$. Define a function $\Gamma:\mathcal{A} \rightarrow \mathcal{P}(\mathcal{A})$ such that for $A \in \mathcal{A}$, $A$ is independent from all events in $\mathcal{A} \setminus (\Gamma(A) \cup \{A\})$. If, for each $0 \le i < m$, we have that both:
\begin{enumerate}
\item $\Pr(A_i) \le \frac{1}{8}$
\item $\sum_{A_j \in \Gamma(A_i)} \Pr(A_j) \le \frac{1}{4}$
\end{enumerate}
then the probability that none of the events occur is positive.
\end{theorem}

For any particular hypergraph, using these more general versions of the local lemma can require establishing complicated bounds on the probabilities and size of sets of dependent events which are different when the hypergraph is changed. We choose to focus on one specific hypergraph to highlight the challenges and benefits of this approach. Let $H_{15}$ be the hypergraph with vertex set $\{a,b,c,d,e,f,g,h,i,j,k,0,1,2,3\}$ and edge set containing the four hyperedges of rank 3 from $\{0,1,2,3\}$ and all possible edges of rank 2 that are not contained within any  hyperedge of rank 3.  Cohen et al.~\cite{cohen2003vsi} previously studied this hypergraph and other similar ones 
and used simulated annealing to construct variable strength covering arrays over them.

With $n$ rows, the probability of missing coverage on any pair of columns is
\[
  p(A) \le v^2\left(\frac{v^2-1}{v^2}\right)^n.
\]
The probability of missing coverage on one of the four sets of three columns is
\[
  p(A) \le v^3\left(\frac{v^3-1}{v^3}\right)^n.
\]
To determine the sets of dependent events we classify the hyperedges into three kinds.  $E_1$ is the set of pairs from $\{a,b,c,d,e,f,g,h,i,j,k\}$, $E_2$ is the set of pairs from $\{a,b,c,d,e,f,g,h,i,j,k\} \times \{0,1,2,3\}$ and $E_3$ is the set of triples from $\{0,1,2,3\}$.  The number of events of each type that are dependent on other types is summarized in Table~\ref{tbl-prob-asymexdeps}.  The entry in row $E_i$ and column $E_j$ is the number of $E_j$ events that are dependent on an $E_i$ event.

\begin{table}[ht]
\begin{center}
\begin{tabular}{|r|r|r|r|} \hline
& $E_1$ & $E_2$ & $E_3$ \\ \hline
$E_1$ & $18$ & $8$ & 0 \\ \hline
$E_2$ & $10$ & $13$ & $3$ \\ \hline
$E_3$ & 0 & $33$ & $3$ \\ \hline
\end{tabular}
\caption{\label{tbl-prob-asymexdeps} Bad event dependency counts on hypergraph $H_{15}$.}
\end{center}
\end{table}

We start by applying the symmetric local lemma for this hypergraph.  For  $A_1 \in E_1$, $|\Gamma(A_1)| = 26$; for $A_2 \in E_2$, $|\Gamma(A_2)| = 26$; and for $A_3 \in E_3$, $|\Gamma(A_3)| = 36$. Thus all bad events are avoided when:
\[
n \geq \frac{3 \ln v + \ln 37 + 1}{3\ln v - \ln(v^3-1)}.
\]
These values of $n$ are given in Table~\ref{tbl-prob-exbinsearch} in column $n_s$ for $v= 2, \ldots, 10$.

To use the general form of the local lemma, as in Theorem~\ref{thm-prob-generallocallemma}, we need to find a function $x\colon \mathcal{A} \rightarrow (0,1)$ such that, for each $A \in \mathcal{A}$:
\begin{equation}
  \Pr(A) \le x(A) \prod_{B \in \Gamma(A)} (1 - x(B)). \label{eqn-gen-prob}
\end{equation}
If such a function $x$ exists, then the local lemma guarantees that the probability that all bad events can be avoided is nonzero. We look for a function $x$ that is fixed on each $E_i$.  Equation~(\ref{eqn-gen-prob}) gives the following system of inequalities:  
\begin{align*}
v^2\left(\frac{v^2-1}{v^2}\right)^n &\le x_1 (1-x_1)^{18} (1-x_2)^8,  \\
v^2\left(\frac{v^2-1}{v^2}\right)^n &\le x_2 (1-x_1)^{10} (1-x_2)^{13} (1-x_3)^3,  \\
v^3\left(\frac{v^3-1}{v^3}\right)^n &\le x_3 (1-x_2)^{33} (1-x_3)^3. 
\end{align*}
No closed form is apparent so we solved this system numerically using OpenOpt \cite{openopt}.  We provide the results in column $n_g$ of Table~\ref{tbl-prob-exbinsearch}.  For details of the solution technique see~\cite{raaphorstphd}.

Since the solutions for the system of equations from the general local lemma can be quite difficult to produce \cite{habib1998}, we also consider the asymmetric local lemma, as in Theorem~\ref{thm-prob-asymlocallemma}.  This gives the following system of inequalities:
\begin{align*}
v^2\left(\frac{v^2-1}{v^2}\right)^n &\le \frac{1}{4},\\
v^3\left(\frac{v^3-1}{v^3}\right)^n &\le \frac{1}{4},\\
26\left[v^2\left(\frac{v^2-1}{v^2}\right)^n\right] & \le \frac{1}{8}, \\
33\left[v^2\left(\frac{v^2-1}{v^2}\right)^n\right] + 3\left[v^3\left(\frac{v^3-1}{v^3}\right)^n\right] & \le \frac{1}{8}.
\end{align*}
which can be solved more directly.  Details can be found in \cite{raaphorstphd}.  The results are given in column $n_a$ of Table~\ref{tbl-prob-exbinsearch}.

\begin{table}[t]
\begin{center}
\begin{tabular}{|r|r|r|r|r|r|} \hline
$v$ & $n_{s}$ & $n_g$ & $n_a$ & $p_{g,s}$  (\%)& $p_{a,s}$  (\%)\\ \hline
2  & 50.10    & 33.79   & 34.38   & 32.55 & 31.37 \\
3  & 209.50   & 148.30  & 153.17  & 29.21 & 26.88 \\
4  & 556.87   & 407.02  & 421.87  & 26.91 & 24.24 \\
5  & 1175.17  & 881.51  & 910.49  & 24.99 & 22.52 \\
6  & 2152.02  & 1643.10 & 1693.86 & 23.65 & 21.28 \\
7  & 3578.65  & 2777.33 & 2850.50 & 22.39 & 20.34 \\
8  & 5549.38  & 4367.67 & 4461.93 & 21.29 & 19.59 \\
9  & 8161.08  & 6440.68 & 6612.28 & 21.08 & 18.97 \\
10 & 11512.91 & 9171.64 & 9387.96 & 20.34 & 18.45 \\ \hline
\end{tabular}
\caption{\label{tbl-prob-exbinsearch} Different values of $n$ obtained for $H_{15}$ using the symmetric, general and asymmetric local lemmas.}
\end{center}
\end{table}

Table~\ref{tbl-prob-exbinsearch} gives the values of $n$ obtained for $H_{15}$ using each of the three versions of the local lemma.  Column $p_{g,s}$ gives the percentage improvement of the general local lemma with respect to the symmetric local lemma.  Column $p_{a,s}$ gives the percentage improvement of the asymmetric local lemma with respect to the symmetric local lemma.  The results obtained from the general local lemma showed significant improvement over those from the symmetric local lemma, with an average improvement of 24.71\% and a median improvement of 23.65\%. We note, however, that considerable work went into finding valid functions $x\colon \mathcal{A} \rightarrow (0,1)$ satisfying the conditions in Theorem~\ref{thm-prob-generallocallemma}. The process we used would be hard to automate and was highly intensive, requiring significant manual experimentation and interaction. The asymmetric local lemma gives results that are slightly worse than those given by the general local lemma. On average, they are within 2.76\% (median: 2.66\%) of those given by the general local lemma. While this difference is small, the improvement that the use of the asymmetric local lemma gives over the use of the symmetric local lemma is considerably more significant: the asymmetric local lemma is, on average, an improvement of 22.63\% (median: 21.29\%) over the symmetric local lemma.  We believe that when the hypergraph for a variable strength covering array lacks uniformity with respect to the size of edges or the sizes of their neighbourhoods, there is benefit using the general and asymmetric forms of the local lemma.  

As a final comparison, we compare the size of covering arrays given by the local lemma to the guaranteed bound of \vardens and also to results of running {\sc VarDens}. The results are shown in  Figure~\ref{fig-prob-regression}.  The best local lemma bounds (from the general local lemma) are better than the theoretical guarantee of \vardens ($N_{dens}$) but the actual size of arrays given by running \vardens is even better.


\begin{figure}
\begin{center}
\includegraphics[width=12cm]{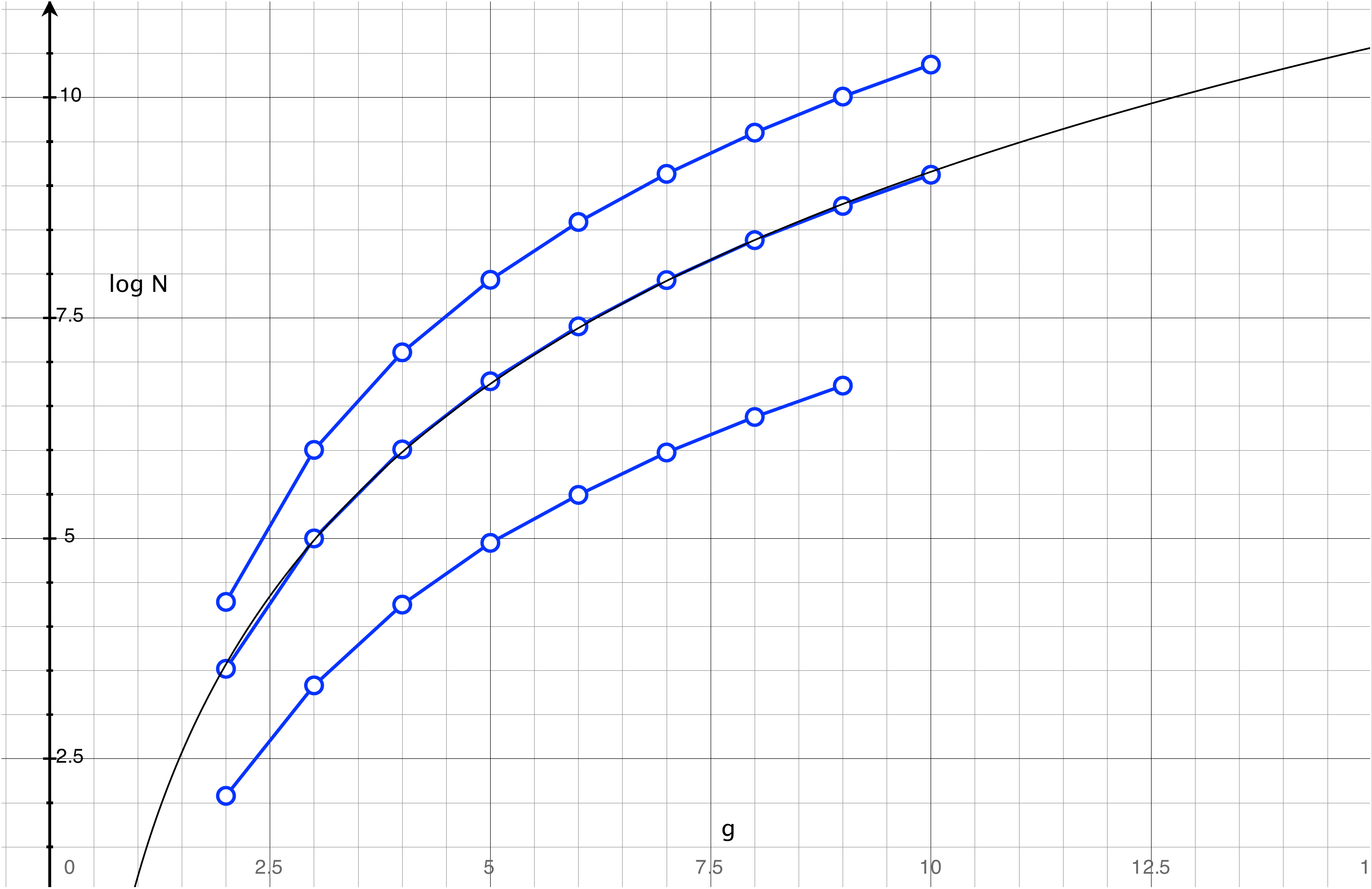}
\caption[Comparison of \vardens and the general local lemma]{\label{fig-prob-regression} A comparison of the \vardens bound (upper line), the general local lemma bounds (middle line), and the actual results from running \vardens on $H_{15}$ (lower line). The horizontal axis is the number of levels and the vertical axis represents the natural logarithm of the size of the arrays.}
\end{center}
\end{figure}

\section{Conclusion}

We used variants of the Lov\'{a}sz local lemma to find upper bounds on the sizes of variable strength covering arrays and compared to the ones constructed using {\sc VarDens}, a derandomized greedy construction, and its guaranteed upper bound.  Our main result, Theorem~\ref{theorem-main}, is a general upper bound on the size of a variable strength covering array in term of the parameters of the associated hypergraph obtained via the symmetric local lemma. When nothing is known about the hypergraph these two bounds are very similar.

The bounds obtained from the symmetric local lemma work best when the hypergraph's edges are of a fixed size and the number of edges intersecting an edge is invariant.
For example, for $t$-designs the local lemma bound is better than the \vardens bound by a constant that depends on $t$. A more extreme example are the cyclic consecutive hypergraphs for which the local lemma bound, unlike \vardens bound, remains constant when we let $k$ grow. We suggest that when edges are of varying sizes, the general and asymmetric versions of the local lemma may work best, and demonstrate this fact with an example. We note that in some instances the \vardens bound is much worse than actual runs of \vardens on given inputs, and in several cases these runs outperform the local lemma bound.

One direction for future research is examining how much of the recent improvements in the application of both the local lemma and density based greedy algorithm for standard covering arrays \cites{MR3328867,francetic2016,sarkar_upper_2017,colbourn_two-stage_nodate_alt,colbourn_asymptotic_nodate} can be extended to variable strength covering arrays.  Another direction is to continue the exploration of the general and asymmetric local lemmas and their utility for variable strength covering arrays. Perhaps some families of hypergraphs are more amenable to the general case and closed form solutions for the required function exist.  The weighted local lemma is also deserving of attention \cite{habib1998}.

\bigskip
\noindent
{\bf Acknowledgements}\\
Lucia Moura and Brett Stevens were  supported by NSERC Discovery grants. The authors would like to thank an anonymous reviewer for various suggestions that improved the presentation of this paper.

\bibliographystyle{elsarticle-num} 

\bibliography{./all.bib}

\begin{bibdiv}
\begin{biblist}

\bib{openopt}{misc}{
       title={Openopt},
         how={\url{http://openopt.org}},
        note={[Online; accessed 2012-07-27]},
}

\bib{bryce09}{article}{
      author={Bryce, R.~C.},
      author={Colbourn, C.~J.},
       title={A density-based greedy algorithm for higher strength covering
  arrays},
        date={2009},
     journal={Softw. Test. Verif. Reliab.},
      volume={19},
      number={1},
       pages={37\ndash 53},
}

\bib{cohen96}{article}{
      author={Cohen, David~M.},
      author={Dalal, Siddhartha~R.},
      author={Parelius, Jesse},
      author={Patton, Gardner~C.},
       title={The combinatorial design approach to automatic test generation},
        date={1996},
     journal={IEEE Software},
      volume={13},
      number={5},
       pages={83\ndash 88},
}

\bib{cohen2003vsi}{inproceedings}{
      author={Cohen, M.~B.},
      author={Colbourn, C.~J.},
      author={Collofello, J.~S.},
      author={Gibbons, P.~B.},
      author={Mugridge, W.~B.},
       title={Variable strength interaction testing of components},
        date={2003},
   booktitle={Proceedings of the 27th international computer software and
  applications conference (compsac 2003)},
     address={Dallas TX},
       pages={413\ndash 418},
}

\bib{colbourn04}{article}{
      author={Colbourn, C.~J.},
       title={Combinatorial aspects of covering arrays},
        date={2004},
        ISSN={0373-3505},
     journal={Matematiche (Catania)},
      volume={59},
      number={1-2},
       pages={125\ndash 172 (2006)},
}

\bib{crchandbook}{book}{
      editor={Colbourn, C.~J.},
      editor={Dinitz, J.~H.},
       title={Handbook of combinatorial designs, second edition},
   publisher={Chapman \& Hall/CRC},
        date={2006},
}

\bib{colbourn_asymptotic_nodate}{unpublished}{
      author={Colbourn, Charles~J.},
      author={Lanus, Erin},
      author={Sarkar, Kaushik},
       title={Asymptotic and constructive methods for covering perfect hash
  families and covering arrays},
        date={2018},
      volume={86},
      number={4},
}

\bib{dalal98}{article}{
      author={Dalal, Siddhartha~R.},
      author={Mallows, Colin~L.},
       title={Factor-covering designs for testing software},
        date={1998August},
     journal={Technometrics},
      volume={40},
      number={3},
       pages={234\ndash 243},
}

\bib{lovasz75}{incollection}{
      author={Erd\H{o}s, P.},
      author={Lov\'{a}sz, L.},
       title={Problems and results on 3-chromatic hypergraphs and some related
  questions},
        date={1975},
   booktitle={Infinite and finite sets},
      volume={11},
   publisher={Coll. Math. Soc. J. Bolyai},
       pages={609\ndash 627},
}

\bib{francetic2016}{article}{
      author={Franceti\'c, Nevena},
      author={Stevens, Brett},
       title={Asymptotic size of covering arrays: an application of entropy
  compression},
        date={2017},
        ISSN={1063-8539},
     journal={J. Combin. Des.},
      volume={25},
      number={6},
       pages={243\ndash 257},
}

\bib{godbole2010cca}{article}{
      author={Godbole, A.~P.},
      author={Koutras, M.~V.},
      author={Milienos, F.~S.},
       title={Consecutive covering arrays and a new randomness test},
        date={2010},
     journal={Journal of Statistical Planning and Inference},
      volume={140},
      number={5},
       pages={1292\ndash 1305},
}

\bib{godbole2011}{article}{
      author={Godbole, A.~P.},
      author={Koutras, M.~V.},
      author={Milienos, F.~S.},
       title={Binary consecutive covering arrays},
        date={2011},
     journal={Annals of the Institute of Statistical Mathematics},
      volume={63},
      number={3},
       pages={559\ndash 584},
}

\bib{godbole96}{article}{
      author={Godbole, A.~P.},
      author={Skipper, D.~E.},
      author={Sunley, R.~A.},
       title={{$t$}-covering arrays: upper bounds and {P}oisson
  approximations},
        date={1996},
        ISSN={0963-5483},
     journal={Combin. Probab. Comput.},
      volume={5},
      number={2},
       pages={105\ndash 117},
}

\bib{habib1998}{book}{
      editor={Habib, Michel},
       title={Probabilistic methods for algorithmic discrete mathematics},
   publisher={Springer},
        date={1998},
}

\bib{1401.3665}{article}{
      author={{Keevash}, Peter},
       title={{The existence of designs}},
        date={2014-01},
     journal={arXiv e-prints},
       pages={arXiv:1401.3665},
      eprint={https://arxiv.org/abs/1401.3665},
}

\bib{kuhn04}{article}{
      author={Kuhn, D.~Richard},
      author={Wallace, Dolores~R.},
      author={{Gallo Jr.}, Albert~M.},
       title={Software fault interactions and implications for software
  testing},
        date={2004},
     journal={IEEE Transactions on Software Engineering},
      volume={30},
      number={6},
       pages={418\ndash 421},
}

\bib{moser2010}{article}{
      author={Moser, R.~A.},
      author={Tardos, G.},
       title={A constructive proof of the general {L}ov\'asz local lemma},
        date={2010},
        ISSN={0004-5411},
     journal={J. ACM},
      volume={57},
      number={2},
       pages={Art. 11, 15},
}

\bib{cai-ea}{article}{
      author={Moura, Lucia},
      author={Raaphorst, Sebastian},
      author={Stevens, Brett},
       title={The {L}ov\'{a}sz local lemma and variable strength covering
  arrays},
        date={2018},
     journal={Electron. Notes Discrete Math.},
      volume={65},
       pages={43\ndash 49},
        note={7th {I}nternational {C}onference on {A}lgebraic {I}nformatics
  ({CAI} 2017): {D}esign {T}heory {T}rack},
}

\bib{raaphorstphd}{thesis}{
      author={Raaphorst, S.},
       title={Variable strength covering arrays},
        type={Ph.D. Thesis},
        date={2012},
        note={University of Ottawa},
}

\bib{vcadbga}{unpublished}{
      author={Raaphorst, S.},
      author={Moura, L.},
      author={Stevens, B.},
       title={A density-based greedy algorithm for variable strength covering
  arrays},
        note={preprint, 24 pages.},
}

\bib{raaphorst2012}{article}{
      author={Raaphorst, Sebastian},
      author={Moura, Lucia},
      author={Stevens, Brett},
       title={A construction for strength-3 covering arrays from linear
  feedback shift register sequences},
        date={2014},
        ISSN={0925-1022},
     journal={Des. Codes Cryptogr.},
      volume={73},
      number={3},
       pages={949\ndash 968},
}

\bib{vca}{unpublished}{
      author={Raaphorst, Sebastian},
      author={Moura, Lucia},
      author={Stevens, Brett},
       title={Variable strength covering arrays},
        date={2018},
      volume={26},
      number={9},
}

\bib{colbourn_two-stage_nodate_alt}{article}{
      author={{Sarkar}, Kaushik},
      author={{Colbourn}, Charles~J.},
       title={{Two-stage algorithms for covering array construction}},
        date={2016-06},
     journal={arXiv e-prints},
       pages={arXiv:1606.06730},
      eprint={{https://arxiv.org/abs/1606.06730}},
        note={submitted to {\em Discrete Applied Math.}},
}

\bib{sarkar_upper_2017}{article}{
      author={Sarkar, Kaushik},
      author={Colbourn, Charles~J.},
       title={Upper bounds on the size of covering arrays},
        date={2017},
        ISSN={0895-4801},
     journal={SIAM Journal on Discrete Mathematics},
      volume={31},
      number={2},
       pages={1277\ndash 1293},
}

\bib{MR3808633}{article}{
      author={Sarkar, Kaushik},
      author={Colbourn, Charles~J.},
      author={De~Bonis, Annalisa},
      author={Vaccaro, Ugo},
       title={Partial covering arrays: algorithms and asymptotics},
        date={2018},
        ISSN={1432-4350},
     journal={Theory Comput. Syst.},
      volume={62},
      number={6},
       pages={1470\ndash 1489},
}

\bib{spencer77}{article}{
      author={Spencer, J.},
       title={Asymptotic lower bounds for {R}amsey functions},
        date={1977/78},
        ISSN={0012-365X},
     journal={Discrete Math.},
      volume={20},
      number={1},
       pages={69\ndash 76},
}

\bib{MR3433921}{article}{
      author={Tzanakis, Georgios},
      author={Moura, Lucia},
      author={Panario, Daniel},
      author={Stevens, Brett},
       title={Constructing new covering arrays from {LFSR} sequences over
  finite fields},
        date={2016},
        ISSN={0012-365X},
     journal={Discrete Math.},
      volume={339},
      number={3},
       pages={1158\ndash 1171},
  url={https://doi-org.proxy.library.carleton.ca/10.1016/j.disc.2015.10.040},
}

\bib{MR3328867}{article}{
      author={Yuan, R.},
      author={Koch, Z.},
      author={Godbole, A.},
       title={Covering array bounds using analytical techniques},
        date={2014},
        ISSN={0384-9864},
     journal={Congr. Numer.},
      volume={222},
       pages={65\ndash 73},
}

\end{biblist}
\end{bibdiv}





\end{document}